\documentclass[11pt,letterpaper,reqno]{amsart}

\usepackage{amssymb}
\usepackage{amsmath}
\usepackage{amsthm}
\usepackage{amsfonts}
\usepackage{bbm}
\usepackage{enumitem}
\usepackage{booktabs}

\usepackage{graphicx}
\usepackage[T1]{fontenc}
\usepackage{doi}
\usepackage{float}

\usepackage{tikz}
\usetikzlibrary{positioning, shapes.geometric, arrows.meta}
\usepackage{pgfplots}
\pgfplotsset{compat=1.18}

\usepackage{bookmark}
\usepackage{hyperref}
\hypersetup{pdfstartview={FitH}}

\numberwithin{equation}{section}
\allowdisplaybreaks

\theoremstyle{plain}
\newtheorem{thm}{Theorem}[section]
\newtheorem{lem}[thm]{Lemma}
\newtheorem{prop}[thm]{Proposition}
\newtheorem{cor}[thm]{Corollary}

\newtheorem{prob}[thm]{Problem}

\theoremstyle{definition}

\theoremstyle{remark}
\newtheorem{rem}[thm]{Remark}

\DeclareMathOperator{\Tr}{Tr}

\newcommand{\A}{\mathcal{A}}
\newcommand{\Apos}{\A_{+}}

\newcommand{\nablaop}{\mathbin{\nabla}}

\newcommand{\kappaop}[1]{\mathbin{\kappa_{#1}}}

\newcommand{\unorm}[1]{%
	\left|\mkern-1.5mu\left|\mkern-1.5mu\left|#1\right|
	\mkern-1.5mu\right|\mkern-1.5mu\right|%
}

\begin{document}
	
	\title[Trace arithmetic--$\kappa_p$ inequality]{Trace arithmetic--$\kappa_p$ inequality}
	
	\author[T.~Zhang]{Teng Zhang}
	\address{School of Mathematics and Statistics, Xi'an Jiaotong University, Xi'an 710049, P.~R. China}
	\email{teng.zhang@stu.xjtu.edu.cn}
	
	\subjclass[2020]{15A60, 46L52, 47A63}
	\keywords{unitarily invariant norm, noncommutative H\"older inequality, faithful trace, $C^\ast$-algebra, Bures--Hellinger-type distances}
	
	\begin{abstract}
		Let $\A$ be a unital $C^\ast$-algebra equipped with a faithful tracial positive linear functional $\tau$. Denote by $\Apos$ its positive cone.
		For $p>0$ and $A,B\in\Apos$, we consider the operations
		\[
		A\kappaop{p}B:=\bigl(A^{p/4}B^{p/2}A^{p/4}\bigr)^{1/p},
		\qquad
		A\nablaop B:=\frac{A+B}{2}.
		\]
	We prove that, for all $p>0$ and all $A,B\in\Apos$,
	\[
	\tau(A\kappaop{p}B)\le \sqrt{\tau(A)\tau(B)}\le \tau(A\nablaop B),
	\]
	thereby answering \cite[Problem~1]{KM24}, posed by \'A.~Kom\'alovics and L.~Moln\'ar, in the affirmative.
		We also record a unitarily invariant norm analogue of the key estimate in the matrix case, and we provide explicit $2\times2$
		counterexamples showing that the triangle inequality for $d_p$ may fail when $0<p<1$ (already for $p=\tfrac12$),
		giving a partial answer to \cite[Problem~2]{KM24}.
	\end{abstract}
	
	\maketitle
	
	\section{Introduction}
	
	Let $\A$ be a unital $C^\ast$-algebra, and denote by $\Apos$ its positive cone.
	For $p>0$ and $A,B\in\Apos$, define
	\begin{equation}\label{eq:def-kappa}
		A\kappaop{p}B:=\bigl(A^{p/4}B^{p/2}A^{p/4}\bigr)^{1/p},
	\end{equation}
	and the arithmetic mean
	\begin{equation*}
		A\nablaop B:=\frac{A+B}{2}.
	\end{equation*}
	Both expressions are well defined for $A,B\in\Apos$ by the continuous functional calculus on $[0,\infty)$.
	In the matrix setting, let $\mathbb{M}_n(\mathbb{C})$ denote the algebra of all $n\times n$ complex matrices, and let
	$\mathbb{M}_n(\mathbb{C})_{+}$ be its cone of positive semidefinite matrices. For $A\in\mathbb{M}_n(\mathbb{C})$,
	we write $\Tr(A)$ for the trace of $A$.
	
	Assume now that $\tau$ is a faithful tracial positive linear functional on $\A$.
	Following Kom\'alovics and Moln\'ar \cite{KM24}, for $p>0$ we set
	\begin{equation}\label{eq:def-dptau}
		d_p^\tau(A,B):=\bigl(\tau(A\nablaop B-A\kappaop{p}B)\bigr)^{1/2},
		\qquad A,B\in\Apos.
	\end{equation}

	In the matrix case $\A=\mathbb{M}_n(\mathbb{C})$ with $\tau=\Tr$, the quantities $\sqrt2\,d_1^{\Tr}$ and $\sqrt2\,d_2^{\Tr}$
	coincide with the quantum Hellinger distance and the Bures (or Bures--Wasserstein) distance, respectively; see, for instance,
	\cite{BGJ19,SIOR17}.
	
	Kom\'alovics and Moln\'ar \cite{KM24} showed that $d_p^\tau(A,B)$ is well defined and nonnegative for $0<p\le 2$ on any $C^\ast$-algebra
	admitting a faithful trace, using the Araki--Lieb--Thirring inequality (see, e.g., \cite{Ara90}).
	In the finite-dimensional case $\A=\mathbb{M}_n(\mathbb{C})$, they further established nonnegativity for $p>2$ by combining a monotonicity
	argument in $p$ with a finite-dimensional analysis of the limit $p\to\infty$ (see \cite{AH14}).
	They pointed out that, for general $C^\ast$-algebras, the case $p>2$ is more delicate precisely because such finite-dimensional tools are unavailable,
	and they posed the following problem.
	
	\begin{prob}[{\cite[Problem~1]{KM24}}]\label{prob:KM1}
		Is it true that for every $C^\ast$-algebra $\A$ with a faithful tracial positive linear functional $\tau$,
		\[
		\tau(A\nablaop B-A\kappaop{p}B)\ge 0
		\]
		holds for all $A,B\in\Apos$ and all real numbers $p>2$?
	\end{prob}
	
	They also asked about the metric property of \eqref{eq:def-dptau}:
	
	\begin{prob}[{\cite[Problem~2]{KM24}}]\label{prob:KM2}
		Let $n\ge2$. Is $d_p^{\Tr}$ on $\mathbb{M}_n(\mathbb{C})_{+}$ a metric for each $1<p<2$,
		and is it \emph{not} a metric for each $0<p<1$?
		More generally, if $\A$ is a $C^\ast$-algebra with a faithful trace $\tau$, is $d_p^\tau$ a metric for $1<p<2$?
	\end{prob}
	
	In this paper, we answer Problem~\ref{prob:KM1} in the affirmative. In fact, we prove the desired trace inequality for all $p>0$.
	The proof  relies  on the noncommutative H\"older inequality and a standard embedding lemma.
	In addition, in the matrix case, we establish an unitarily invariant norm inequality
	which yields the trace estimate as a special case.
	We also provide explicit $2\times2$ matrices showing that the triangle inequality fails for $p=\tfrac12$, giving a concrete counterexample
	in the range $0<p<1$ in Problem~\ref{prob:KM2}.
	
	\section{The matrix case and a unitarily invariant norm inequality}
	
	For $X\in \mathbb{M}_n(\mathbb{C})$, set $|X|=(X^\ast X)^{1/2}$.
	A norm $\unorm{\cdot}$ on $\mathbb{M}_n(\mathbb{C})$ is said to be \emph{unitarily invariant} if
	$\unorm{UXV}=\unorm{X}$ for all unitary matrices $U,V$.
	
	Our starting point is the following H\"older-type inequality for unitarily invariant norms.
	
	\begin{lem}[{\cite[Exercise~IV.2.7]{Bha97}}]\label{lem:UIHolder}
		Let $\alpha,\beta,\gamma>0$ satisfy $\frac1\alpha+\frac1\beta=\frac1\gamma$.
		Then for every unitarily invariant norm $\unorm{\cdot}$ on $\mathbb{M}_n(\mathbb{C})$,
		\[
		\unorm{|AB|^{\,\gamma}}^{1/\gamma}\le \unorm{|A|^{\,\alpha}}^{1/\alpha}\,\unorm{|B|^{\,\beta}}^{1/\beta},
		\qquad A,B\in\mathbb{M}_n(\mathbb{C}).
		\]
	\end{lem}
	
	We next apply Lemma~\ref{lem:UIHolder} to the operator mean $A\kappaop{p}B$ and obtain a norm inequality that will later yield,
	as a special case, the desired trace estimate in the matrix setting.
	
	\begin{thm}\label{thm:matrix-uinorm}
		Let $p>0$ and $A,B\in\mathbb{M}_n(\mathbb{C})_{+}$. Then
		\begin{equation}\label{eq:matrix-main-uinorm}
			\unorm{A\kappaop{p}B}\le \unorm{A}^{1/2}\unorm{B}^{1/2}\le \frac{\unorm{A}+\unorm{B}}{2}.
		\end{equation}
	\end{thm}
	
	\begin{proof}
		Set $X:=A^{p/4}B^{p/4}$. Then $XX^\ast=A^{p/4}B^{p/2}A^{p/4}$, hence
		\[
		A\kappaop{p}B=(XX^\ast)^{1/p}.
		\]
		Let $X=U|X|$ be the polar decomposition. Then $XX^\ast=U|X|^2U^\ast$, and the functional calculus gives
		$(XX^\ast)^{1/p}=U|X|^{2/p}U^\ast$.
		By unitary invariance,
		\begin{equation}\label{eq:uinorm-polar}
			\unorm{A\kappaop{p}B}=\unorm{|X|^{2/p}}=\unorm{\bigl|A^{p/4}B^{p/4}\bigr|^{2/p}}.
		\end{equation}
		
		Apply Lemma~\ref{lem:UIHolder} with $\gamma=2/p$ and $\alpha=\beta=4/p$ to $A^{p/4}$ and $B^{p/4}$. We obtain
		\[
		\unorm{\bigl|A^{p/4}B^{p/4}\bigr|^{2/p}}^{p/2}
		\le
		\unorm{\bigl|A^{p/4}\bigr|^{4/p}}^{p/4}\,
		\unorm{\bigl|B^{p/4}\bigr|^{4/p}}^{p/4}
		=
		\unorm{A}^{p/4}\unorm{B}^{p/4}.
		\]
		Raising both sides to the power $2/p$ and using \eqref{eq:uinorm-polar} yields the first inequality in \eqref{eq:matrix-main-uinorm}.
		The second inequality follows from the scalar AM--GM inequality $\sqrt{uv}\le (u+v)/2$ for $u,v\ge 0$, applied to
		$u=\unorm{A}$ and $v=\unorm{B}$.
	\end{proof}
	
	As an immediate consequence, choosing the trace norm gives the corresponding trace inequality.
	
	\begin{cor}\label{cor:matrix-trace}
		Let $p>0$ and $A,B\in\mathbb{M}_n(\mathbb{C})_{+}$. Then
		\begin{equation}\label{eq:matrix-main-trace}
			\Tr(A\kappaop{p}B)\le \sqrt{\Tr(A)\Tr(B)}\le \Tr(A\nablaop B).
		\end{equation}
		In particular, $\Tr(A\nablaop B-A\kappaop{p}B)\ge 0$ for all $p>0$.
	\end{cor}
	
	\begin{proof}
		Apply Theorem~\ref{thm:matrix-uinorm} to the trace norm $\|X\|_1=\Tr|X|$.
		Since $A\kappaop{p}B\ge 0$, we have $\|A\kappaop{p}B\|_1=\Tr(A\kappaop{p}B)$, and similarly
		$\|A\|_1=\Tr(A)$ and $\|B\|_1=\Tr(B)$.
		Thus \eqref{eq:matrix-main-uinorm} reduces to \eqref{eq:matrix-main-trace}.
	\end{proof}
	
	\section{$C^\ast$-algebras with faithful traces: solution of Problem~\ref{prob:KM1}}
	
	We now pass from matrices to general $C^\ast$-algebras. The argument proceeds in two steps: first, we embed $(\A,\tau)$ into a finite von Neumann
	algebra so that the trace is preserved; second, we apply H\"older's inequality in the associated noncommutative $L^p$-spaces.
	
	\begin{lem}[{\cite[Lemma~3]{KM24}}]\label{lem:KM-lemma3}
		Let $\A$ be a $C^\ast$-algebra equipped with a faithful tracial positive linear functional $\tau$.
		Then there exist a von Neumann algebra $M$ endowed with a faithful normal tracial positive linear functional $\nu$
		and an isometric $^\ast$-isomorphism $\varphi$ from $\A$ onto an ultraweakly (equivalently, $\sigma$-weakly) dense $C^\ast$-subalgebra of $M$
		such that $\nu(\varphi(X))=\tau(X)$ for all $X\in\A$.
		In particular, after normalizing $\nu$ to a tracial state, $M$ is a finite von Neumann algebra in the usual sense.
	\end{lem}
	
	\begin{lem}[{\cite[p.~1464]{PS03}}]\label{lem:nc-Holder}
		Let $(M,\nu)$ be a finite von Neumann algebra with a faithful normal finite trace $\nu$ (not necessarily normalized).
		For $0<s<\infty$ and $X\in L_s(M,\nu)$, set $\|X\|_s=(\nu(|X|^s))^{1/s}$.
		If $0<p,q,r<\infty$ satisfy $\frac1r=\frac1p+\frac1q$, then
		\[
		\|XY\|_r\le \|X\|_p\,\|Y\|_q,
		\qquad X\in L_p(M,\nu),\ Y\in L_q(M,\nu).
		\]
		When $s<1$, $\|\cdot\|_s$ is only a quasi-norm; nevertheless, the above H\"older inequality remains valid.
	\end{lem}
	
	Combining these two ingredients yields the desired trace inequality in full generality.
	
	\begin{thm}\label{thm:Cstar}
		Let $\A$ be a unital $C^\ast$-algebra with a faithful tracial positive linear functional $\tau$.
		Then for every $p>0$ and all $A,B\in\Apos$,
		\begin{equation}\label{eq:Cstar-main}
			\tau(A\kappaop{p}B)\le \sqrt{\tau(A)\tau(B)}\le \tau(A\nablaop B).
		\end{equation}
		In particular, $\tau(A\nablaop B-A\kappaop{p}B)\ge 0$ for all $p>0$.
	\end{thm}
	\begin{proof}
		By Lemma~\ref{lem:KM-lemma3}, we may (via $\varphi$) identify $\A$ with an ultraweakly dense $C^\ast$-subalgebra of a finite von Neumann algebra $(M,\nu)$,
		in such a way that $\tau=\nu|_{\A}$. Since $\varphi$ respects the continuous functional calculus, it suffices to prove \eqref{eq:Cstar-main} in $(M,\nu)$.
		
		Fix $p>0$ and $A,B\in M_{+}$, and set $X:=A^{p/4}B^{p/4}$. Then $XX^\ast=A^{p/4}B^{p/2}A^{p/4}$, and hence
		\[
		A\kappaop{p}B=(XX^\ast)^{1/p}.
		\]
		Let $X=u|X|$ be the polar decomposition in $M$, where $u$ is a partial isometry. Then $XX^\ast=u|X|^2u^\ast$, so
		$(XX^\ast)^{1/p}=u|X|^{2/p}u^\ast$.
		Using traciality and the fact that $u^\ast u$ is the support projection of $|X|$ (so that $|X|^{2/p}=|X|^{2/p}u^\ast u$), we obtain
		\[
		\nu(A\kappaop{p}B)
		=\nu\bigl((XX^\ast)^{1/p}\bigr)
		=\nu\bigl(u|X|^{2/p}u^\ast\bigr)
		=\nu\bigl(|X|^{2/p}u^\ast u\bigr)
		=\nu\bigl(|X|^{2/p}\bigr)
		=\|X\|_{2/p}^{2/p}.
		\]
		
		Now apply Lemma~\ref{lem:nc-Holder} with exponents $2/p$ and $4/p$:
		\[
		\|A^{p/4}B^{p/4}\|_{2/p}\le \|A^{p/4}\|_{4/p}\,\|B^{p/4}\|_{4/p}.
		\]
		Since $A^{p/4}\ge 0$, we have $\|A^{p/4}\|_{4/p}=(\nu(A))^{p/4}$, and similarly $\|B^{p/4}\|_{4/p}=(\nu(B))^{p/4}$.
		Raising both sides to the power $2/p$ yields
		\[
		\nu(A\kappaop{p}B)=\|X\|_{2/p}^{2/p}\le \sqrt{\nu(A)\nu(B)}.
		\]
		Finally, the scalar AM--GM inequality gives
		$\sqrt{\nu(A)\nu(B)}\le \frac{\nu(A)+\nu(B)}{2}=\nu(A\nablaop B)$, completing the proof.
	\end{proof}
	
	\section{On Problem~\ref{prob:KM2}: an explicit counterexample for $p=\tfrac12$}
	
	In view of Theorem~\ref{thm:Cstar}, the quantity $d_p^\tau(A,B)$ is well defined and nonnegative for all $p>0$.
	Problem~\ref{prob:KM2} asks whether $d_p^{\Tr}$ (respectively, $d_p^\tau$) is a \emph{metric} in certain parameter ranges.
	In this section, we exhibit explicit $2\times2$ matrices showing that the triangle inequality may fail already for $p=\tfrac12$.
	Throughout this section, we work in $\mathbb{M}_n(\mathbb{C})_{++}$, the cone of positive definite matrices, and we keep the notation
	from \eqref{eq:def-kappa} and \eqref{eq:def-dptau}. Namely, for $p>0$ and $X,Y\in\mathbb{M}_n(\mathbb{C})_{++}$, we set
	\[
	X\kappaop{p} Y:=\bigl(X^{p/4}Y^{p/2}X^{p/4}\bigr)^{1/p},
	\qquad
	d_p(X,Y):=\Bigl(\Tr\Bigl(\frac{X+Y}{2}-X\kappaop{p} Y\Bigr)\Bigr)^{1/2}.
	\]
	
	To streamline the computations in the $2\times2$ case, we first record a convenient closed form for the principal square root.
	
	\begin{lem}\label{lem:2x2-sqrt-closed-form}
		Let $M\in\mathbb{M}_2(\mathbb{C})_{++}$. Then its principal square root is given by
		\begin{equation}\label{eq:2x2-sqrt-closed-form}
			M^{1/2}=\frac{M+\sqrt{\det M}\,I}{\sqrt{\Tr M+2\sqrt{\det M}}}.
		\end{equation}
	\end{lem}
	
	\begin{proof}
		Set $\delta:=\sqrt{\det M}>0$ and define $A:=\delta^{-1}M$. Then $A\in\mathbb{M}_2(\mathbb{C})_{++}$ and $\det A=1$.
		Let $B:=I$, so that $\det B=1$ as well. By \cite[Proposition~4.1.12]{Bha07}, for $2\times2$ positive matrices with determinant one, we have
		\[
		A\#B=\frac{A+B}{\sqrt{\det(A+B)}}.
		\]
		Taking $B=I$ and using $A\#I=A^{1/2}$ yields
		\[
		A^{1/2}=\frac{A+I}{\sqrt{\det(A+I)}}.
		\]
		Substituting $A=\delta^{-1}M$ gives
		\[
		\frac{M^{1/2}}{\sqrt{\delta}}
		=\frac{\delta^{-1}(M+\delta I)}{\sqrt{\det\!\bigl(\delta^{-1}(M+\delta I)\bigr)}}
		=\frac{M+\delta I}{\sqrt{\det(M+\delta I)}}.
		\]
		Hence
		\begin{equation}\label{eq:intermediate}
			M^{1/2}=\sqrt{\delta}\,\frac{M+\delta I}{\sqrt{\det(M+\delta I)}}.
		\end{equation}
		
		It remains to compute $\det(M+\delta I)$. For any $2\times2$ matrix $M$ and any scalar $t$, one has
		\[
		\det(M+tI)=t^2+t\,\Tr M+\det M,
		\]
		for instance by the characteristic polynomial of $M$.
		With $t=\delta$ and $\det M=\delta^2$, we obtain
		\[
		\det(M+\delta I)=\delta^2+\delta\,\Tr M+\delta^2=\delta\bigl(\Tr M+2\delta\bigr).
		\]
		Therefore $\sqrt{\det(M+\delta I)}=\sqrt{\delta}\,\sqrt{\Tr M+2\delta}$, and \eqref{eq:intermediate} becomes
		\[
		M^{1/2}
		=\sqrt{\delta}\,\frac{M+\delta I}{\sqrt{\delta}\,\sqrt{\Tr M+2\delta}}
		=\frac{M+\delta I}{\sqrt{\Tr M+2\delta}},
		\]
		which is exactly \eqref{eq:2x2-sqrt-closed-form}.
	\end{proof}
	
	With Lemma~\ref{lem:2x2-sqrt-closed-form} in hand, we can now construct a concrete triple for which the triangle inequality fails.
	
	\begin{prop}\label{prop:counterexample-p12}
		There exist $A,B,C\in\mathbb{M}_2(\mathbb{C})_{++}$ such that
		\[
		d_{1/2}(A,B)>d_{1/2}(A,C)+d_{1/2}(C,B).
		\]
	\end{prop}
	
	\begin{proof}
		Set
		\[
		A=\begin{pmatrix}4&0\\[2pt]0&1\end{pmatrix},\qquad
		B=\begin{pmatrix}\frac52&\frac32\\[2pt]\frac32&\frac52\end{pmatrix},\qquad
		C=\begin{pmatrix}\frac{25}{8}&\frac34\\[2pt]\frac34&\frac{13}{8}\end{pmatrix}.
		\]
		All three matrices are real symmetric and positive definite:
		\[
		\det A=4,\qquad
		\det B=\Bigl(\frac52\Bigr)^2-\Bigl(\frac32\Bigr)^2=4,\qquad
		\det C=\frac{25}{8}\cdot\frac{13}{8}-\Bigl(\frac34\Bigr)^2=\frac{289}{64}>0.
		\]
		
		Fix $p=\tfrac12$. Then
		\[
	A\kappaop{1/2}B=\bigl(A^{1/8}B^{1/4}A^{1/8}\bigr)^2
		\]
		and hence
		\[
		d_{1/2}(A,B)^2
		=\Tr\Bigl(\frac{A+B}{2}\Bigr)-\Tr\Bigl(\bigl(A^{1/8}B^{1/4}A^{1/8}\bigr)^2\Bigr).
		\]
		
		\medskip
		\noindent\emph{Step 1: exact computation of $d_{1/2}(A,B)^2$.}
		Since $A=\mathrm{diag}(4,1)$,
		\[
		A^{1/8}=\mathrm{diag}(4^{1/8},1)=\mathrm{diag}(2^{1/4},1).
		\]
		Set $s:=2^{1/4}$, so that $s^2=\sqrt2$. Let
		\[
		U:=\frac{1}{\sqrt2}\begin{pmatrix}1&1\\[2pt]1&-1\end{pmatrix}.
		\]
		Then $U$ is orthogonal and one checks that
		\[
		B=U\begin{pmatrix}4&0\\[2pt]0&1\end{pmatrix}U^{\mathsf T},
		\qquad
		B^{1/4}
		=U\begin{pmatrix}\sqrt2&0\\[2pt]0&1\end{pmatrix}U^{\mathsf T}
		=\frac12\begin{pmatrix}\sqrt2+1&\sqrt2-1\\[2pt]\sqrt2-1&\sqrt2+1\end{pmatrix}.
		\]
		Define $P:=A^{1/8}B^{1/4}A^{1/8}$. Then
		\[
		P=\frac12
		\begin{pmatrix}
			s^2(\sqrt2+1) & s(\sqrt2-1)\\[2pt]
			s(\sqrt2-1) & \sqrt2+1
		\end{pmatrix}
		=\frac12
		\begin{pmatrix}
			2+\sqrt2 & s(\sqrt2-1)\\[2pt]
			s(\sqrt2-1) & 1+\sqrt2
		\end{pmatrix}.
		\]
		Since $P$ is symmetric,
		\[
		\Tr(P^2)=P_{11}^2+2P_{12}^2+P_{22}^2=\frac14+3\sqrt2.
		\]
		Moreover,
		\[
		\Tr\Bigl(\frac{A+B}{2}\Bigr)=\frac{\Tr A+\Tr B}{2}=\frac{5+5}{2}=5.
		\]
		Therefore
		\[
		d_{1/2}(A,B)^2
		=5-\Tr(P^2)
		=5-\Bigl(\frac14+3\sqrt2\Bigr)
		=\frac{19}{4}-3\sqrt2.
		\]
		
		\medskip
		\noindent\emph{Step 2: exact computation of $d_{1/2}(A,C)^2$ and $d_{1/2}(B,C)^2$.}
		First note that
		\[
		\sqrt A=\begin{pmatrix}2&0\\[2pt]0&1\end{pmatrix},
		\qquad
		\sqrt B
		=U\begin{pmatrix}2&0\\[2pt]0&1\end{pmatrix}U^{\mathsf T}
		=\begin{pmatrix}\frac32&\frac12\\[2pt]\frac12&\frac32\end{pmatrix}.
		\]
		Define
		\[
		M_0:=\frac{\sqrt A+\sqrt B}{2}
		=\begin{pmatrix}\frac74&\frac14\\[2pt]\frac14&\frac54\end{pmatrix}.
		\]
		A direct multiplication yields
		\[
		M_0^2=\begin{pmatrix}\frac{25}{8}&\frac34\\[2pt]\frac34&\frac{13}{8}\end{pmatrix}=C,
		\]
		so $C=M_0^2$ and hence $C^{1/4}=M_0^{1/2}$.
		
		Applying Lemma~\ref{lem:2x2-sqrt-closed-form} to $M_0$, we have
		\[
		C^{1/4}=M_0^{1/2}
		=\frac{M_0+\sqrt{\det M_0}\,I}{\sqrt{\Tr M_0+2\sqrt{\det M_0}}}.
		\]
		Here $\Tr M_0=3$ and $\det M_0=\frac{17}{8}$, and thus
		\[
		C^{1/4}
		=\frac{M_0+\sqrt{17/8}\,I}{\sqrt{3+2\sqrt{17/8}}}.
		\]
		
		Now set $Q:=A^{1/8}C^{1/4}A^{1/8}$. Then $A\kappaop{1/2}C=Q^2$ and
		\[
		d_{1/2}(A,C)^2=\Tr\Bigl(\frac{A+C}{2}\Bigr)-\Tr(Q^2).
		\]
		A direct substitution yields
		\[
		\Tr(Q^2)=\frac{29}{8}-\frac{\sqrt{17}}{4}+\frac{3\sqrt2}{4}+\frac{3\sqrt{34}}{16},
		\qquad
		\Tr\Bigl(\frac{A+C}{2}\Bigr)=\frac{\Tr A+\Tr C}{2}=\frac{5+19/4}{2}=\frac{39}{8}.
		\]
		Therefore
		\[
		d_{1/2}(A,C)^2
		=\frac{39}{8}-\Tr(Q^2)
		=\frac54+\frac{\sqrt{17}}{4}-\frac{3\sqrt2}{4}-\frac{3\sqrt{34}}{16}.
		\]
		
		Since $B=UAU^{\mathsf T}$ and (as one checks directly) $C=UCU^{\mathsf T}$, the unitary invariance of the trace together with the functional calculus implies
		$d_{1/2}(B,C)=d_{1/2}(A,C)$.
		
		\medskip
		\noindent\emph{Step 3: the triangle inequality fails strictly.}
		Since $d_{1/2}(B,C)=d_{1/2}(A,C)$, the triangle inequality
		\[
		d_{1/2}(A,B)\le d_{1/2}(A,C)+d_{1/2}(C,B)
		\]
		is equivalent to $d_{1/2}(A,B)\le 2d_{1/2}(A,C)$.
		Squaring (both sides are nonnegative) yields
		\[
		d_{1/2}(A,B)^2\le 4d_{1/2}(A,C)^2.
		\]
		Using the explicit formulas above, we compute
		\[
		d_{1/2}(A,B)^2-4d_{1/2}(A,C)^2
		=\frac{3\sqrt{34}}{4}-\sqrt{17}-\frac14.
		\]
		It remains to show that this quantity is positive. Indeed,
		\[
		\frac{3\sqrt{34}}{4}-\sqrt{17}-\frac14>0
		\quad\Longleftrightarrow\quad
		3\sqrt{34}>4\sqrt{17}+1.
		\]
		Squaring (both sides are positive) gives
		\[
		306>(4\sqrt{17}+1)^2=273+8\sqrt{17}
		\quad\Longleftrightarrow\quad
		33>8\sqrt{17}.
		\]
		Squaring again yields $33^2>64\cdot 17$, i.e.\ $1089>1088$, which is true. Hence
		\[
		d_{1/2}(A,B)>2d_{1/2}(A,C)=d_{1/2}(A,C)+d_{1/2}(C,B),
		\]
		so the triangle inequality fails.
	\end{proof}
	
	\begin{rem}
		Proposition~\ref{prop:counterexample-p12} shows that $d_p$ need not be a metric for $0<p<1$.
		It does not address the case $1<p<2$, which remains open in general.
	\end{rem}
	
	\section*{Declaration of competing interest}
	The author declares no competing interests.

	\section*{Data availability}
	No data was used for the research described in the article.

	\section*{Acknowledgments}
	The author thanks his advisor L.~Moln\'ar for suggesting this problem.
	This work is supported by the China Scholarship Council, the Young Elite Scientists Sponsorship Program for PhD Students
	(China Association for Science and Technology), and the Fundamental Research Funds for the Central Universities at Xi'an
	Jiaotong University (Grant No.~xzy022024045).
	

\end{document}